\documentclass[11pt]{amsart}

\usepackage{amscd}
\usepackage{amsmath, amssymb}
\usepackage{amsfonts}
\newcommand{\de}{\partial}

\newcommand{\ddbar}{\sqrt{-1} \partial \overline{\partial}}

\newcommand{\ov}[1]{\overline{#1}}

\newcommand{\tr}[2]{\textrm{tr}_{#1}{#2}}
\newcommand{\ti}[1]{\tilde{#1}}
\newcommand{\vp}{\varphi}

\newcommand{\ve}{\varepsilon}

\renewcommand{\leq}{\leqslant}
\renewcommand{\geq}{\geqslant}
\renewcommand{\le}{\leqslant}
\renewcommand{\ge}{\geqslant}

\newcommand{\be}{\begin{equation}}
\newcommand{\ee}{\end{equation}}

\begin{document}
\newtheorem{claim}{Claim}
\newtheorem{theorem}{Theorem}[section]
\newtheorem{lemma}[theorem]{Lemma}
\newtheorem{corollary}[theorem]{Corollary}
\newtheorem{proposition}[theorem]{Proposition}
\newtheorem{question}{question}[section]
\newtheorem{conjecture}[theorem]{Conjecture}
\numberwithin{equation}{section}
\theoremstyle{definition}
\newtheorem{remark}[theorem]{Remark}

\title[The ABP estimate and the Calabi-Yau equation]{The Aleksandrov-Bakelman-Pucci estimate and the Calabi-Yau equation}
\author[V. Tosatti]{Valentino Tosatti}
\author[B. Weinkove]{Ben Weinkove}
\address{Department of Mathematics, Northwestern University, 2033 Sheridan Road, Evanston, IL 60208}
\thanks{Partially supported by a Sloan Research Fellowship and NSF grant DMS-1308988 (V.T.), and by NSF grant DMS-1406164 (B.W.). Part of this work was done
while the authors were visiting the Center for Mathematical Sciences and Applications at Harvard University, which they thank for the hospitality.}

\begin{abstract}
We give two applications of the Aleksandrov-Bakelman-Pucci estimate to the Calabi-Yau equation on symplectic four-manifolds. The first is solvability of the equation on the Kodaira-Thurston manifold for certain almost-K\"ahler structures assuming $S^1$-invariance, extending a result of Buzano-Fino-Vezzoni.  The second is to reduce the general case of Donaldson's conjecture to a bound on the measure of a superlevel set of a scalar function.
\end{abstract}

\maketitle

\section{Introduction}

Yau's Theorem \cite{Y} states that one can prescribe the volume form of a K\"ahler metric on a compact K\"ahler manifold $M^n$ within a given cohomology class.  The proof reduces, via a continuity method, to obtaining uniform $C^{\infty}$ \emph{a priori} estimates on a potential function $u$ solving the complex Monge-Amp\`ere equation
\begin{equation} \label{yau}
(\omega + \ddbar u)^n = e^F \omega^n, \quad \omega + \ddbar u >0, \quad \sup_M u =0,
\end{equation}
for a given smooth function $F$.  A key step in Yau's paper \cite{Y} was the acclaimed $L^{\infty}$ estimate of $u$, which he obtained using a Moser iteration argument.

There are now alternative proofs of the $L^{\infty}$ estimate, which have been used to extend Yau's Theorem to different settings \cite{K, B, TW2, TW3, B2, DK, TW5}.  In particular, Cheng and Yau (see \cite[p. 75]{Be}) used the Aleksandrov-Bakelman-Pucci (ABP) estimate to prove an $L^2$ stability result for the complex Monge-Amp\`ere equation, and later B{\l}ocki \cite{B} used this idea to give a new proof of the $L^{\infty}$ estimate.  Recall that the ABP estimate states, roughly speaking, that the infimum of a function on a bounded domain can be controlled in terms of its infimum on the boundary and  the integral of the determinant of its Hessian over the set where the function is convex (see e.g. \cite[Lemma 9.2]{GT}).  B{\l}ocki's argument uses this to reduce the $L^{\infty}$ estimate of $u$ to an $L^1$ bound of $u$.
Recently, Sz\'ekelyhidi \cite{Sz} strengthened this method to deal with equations involving first order derivative terms, and this now has been used to establish $L^{\infty}$ estimates for a large class of Monge-Amp\`ere type equations \cite{STW, CJY, CTW}.

The purpose of this note is to apply the ABP estimate to Donaldson's problem of the Calabi-Yau equation on symplectic 4-manifolds.
Donaldson \cite{D} conjectured the following:

\begin{conjecture}\label{don}
Consider a compact symplectic manifold $(M, \omega)$ of real dimension 4, equipped with an almost complex structure $J$ which is tamed by $\omega$, and with  $\tilde{\omega}$ another symplectic form compatible with $J$, cohomologous to $\omega$ and solving the Calabi-Yau equation
\begin{equation} \label{CYeqn}
\tilde{\omega}^2 = e^F \omega^2,
\end{equation}
for some smooth $F$.  Then there are $C^{\infty}$ \emph{a priori} estimates of $\tilde{\omega}$ depending only on $M, J, \omega$ and $F$.
\end{conjecture}

If true, this would establish an almost-K\"ahler version of Yau's Theorem for $4$-manifolds, and would also have consequences for Donaldson's ``tamed to compatible'' conjecture \cite{D} (partially confirmed by Taubes \cite{Ta}, see also the surveys \cite{DLZ,TW4} and the references therein).

While Conjecture \ref{don} still remains open in general, we prove two consequences of the ABP estimate.  The first is for the Calabi-Yau equation on the Kodaira-Thurston manifold $M = (\textrm{Nil}^3/\Gamma)\times S^1$ where $\textrm{Nil}^3$ is the Heisenberg group of invertible matrices of the form
$$\begin{pmatrix} 1 & x & z \\ 0 & 1 & y \\ 0 & 0 & 1 \end{pmatrix}, \quad \textrm{with } x,y,z \in \mathbb{R},$$
and $\Gamma$ is the subgroup consisting of those elements with integer entries, acting by left multiplication.  $M$ admits symplectic forms, but no K\"ahler metric.  The Calabi-Yau equation on $M$ was first considered in \cite{TW} where it was shown that equation (\ref{CYeqn}) is solvable for certain $(J, \omega)$ assuming a $T^2$ symmetry of the initial data. More cases with $T^2$ symmetry were solved in \cite{FLSV,BFV2, V}, and this was recently improved by Buzano-Fino-Vezzoni \cite{BFV} to the case of $S^1$ symmetry, which we now describe.

$M$ has an $S^1$ family of non-integrable almost complex structures $J_{\theta}$ for $\theta \in [0,2\pi)$ which are compatible with a symplectic form $\omega_{\theta}$ on $M$.  In addition, 
there is an $S^1$ action on $M$ given by translation in the $z$ coordinate which preserves $(J_{\theta}, \omega_{\theta})$ for each $\theta$.  For details, see Section \ref{sectionKT} below.

The result of Buzano-Fino-Vezzoni \cite[Theorem 2]{BFV} is that one can solve the Calabi-Yau equation for any $S^1$-invariant $F$  with the additional assumption
\begin{equation} \label{BFVass}
\cos \theta =0 \quad \textrm{or} \quad \tan \theta \in \mathbb{Q}.
\end{equation}
  We give a new approach using the ABP estimate which removes the assumption (\ref{BFVass}) on $\theta$.  More precisely we prove:

\pagebreak[3]

\begin{theorem} \label{theorem1} Given any $0\leq\theta<2\pi$, and any smooth $S^1$-invariant function $F$ on the Kodaira-Thurston manifold $M$ with $\int_M(e^F-1)\omega_\theta^2=0$, there is a unique symplectic form $\ti{\omega}$ on $M$ which is compatible with $J_\theta$, with $[\ti{\omega}]=[\omega_\theta]$, and solving the Calabi-Yau equation
\begin{equation}\label{BFVCY}
\ti{\omega}^2=e^F\omega_\theta^2.
\end{equation}
\end{theorem}

Of course, this gives in particular a proof of Conjecture \ref{don} on $(M,\omega_\theta,J_\theta)$ when $F$ is $S^1$ invariant.

Following \cite{BFV}, the equation (\ref{BFVCY}) can be written as an elliptic PDE for a scalar function $u$ on the three-torus $T^3$.  The ABP estimate shows that the  $L^{\infty}$ estimate reduces to a simple integral bound for $u$ which can be easily established.  Given this we can  apply the rest of the arguments of \cite{BFV}, which do not need the assumption \eqref{BFVass} on $\theta$, to obtain all the other estimates.

Our second application of the ABP estimate shows that in the general case of Conjecture \ref{don} on a $4$-manifold $(M, \omega, J)$, we can reduce all the estimates to an integral bound of a certain potential function.  The idea of using the ABP estimate in this context was first pointed out to us by Sz\'ekelyhidi \cite{Sz2}.  
We now describe our results more precisely.   We apply the ABP estimate to the ``almost-K\"ahler potential'' $\varphi$ of \cite{W, TWY},  defined by
\begin{equation} \label{akp}
\tilde{\omega} =  \omega + \frac{1}{2}dJd \varphi + da, \quad \ti{\omega} \wedge da=0, \quad \sup_M \varphi=0,
\end{equation}
using the sign convention in \cite{CTW} and where $a$ is a $1$-form (see Section \ref{sectionCYs} for more details).
It was shown in \cite{W,TWY} that $C^{\infty}$ estimates for $\tilde{\omega}$ in (\ref{CYeqn}) follow from an $L^{\infty}$ bound on $\varphi$, and this was reduced in \cite{TWY} to a  bound on $\int_M e^{-\alpha \varphi} \omega^2$ for some $\alpha>0$.  Our result here is that we can further reduce this, via the ABP method, to a much weaker kind of  bound. An $L^p$ bound of $\varphi$, for any $0<p<\infty$, would suffice, but in fact much less than this is needed:

\begin{theorem} \label{theorem2} Let $(M,\omega,J)$ be as in Conjecture \ref{don} and let $\tilde{\omega}$ solve the Calabi-Yau equation (\ref{CYeqn}).
Let $\mathcal{F}:\mathbb{R} \to \mathbb{R}_{\geq 0}$ be any increasing function with $\lim_{x\to+\infty}\mathcal{F}(x)=+\infty$.  Then the function  $\vp$, defined by (\ref{akp}) satisfies the estimate
\begin{equation}\label{bd0}
\mathcal{F}(\|\vp\|_{L^\infty(M)}-1)\leq C\int_M \mathcal{F}(-\vp)\omega^2,
\end{equation}
for a uniform constant $C$ depending only on the background data $M, \omega, J, F$.  Hence, to prove Conjecture \ref{don} it is sufficient to bound $\int_M \mathcal{F}(-\vp)\omega^2$ for some fixed $\mathcal{F}$.

Furthermore, if we fix $q>3/2$, then for a uniform constant $C_q$, the function $\varphi$ satisfies the estimate
\begin{equation} \label{mi}
\| \varphi \|_{L^{\infty}(M)} \le C_q \left( \int_{ \{ \varphi \ge - \lambda\}} \omega^2 \right)^{-q} + \lambda,
\end{equation}
for any $\lambda >0$.
\end{theorem}

The second estimate (\ref{mi})  shows that to prove Conjecture \ref{don} it is sufficient to get a uniform positive lower bound of the measure of the set $\{ \varphi \ge - \lambda \}$ for some (possibly very large) uniform $\lambda$.  To establish (\ref{mi}), we prove the strictly stronger result that $\varphi  - \inf_M \varphi$ is uniformly bounded in $L^p$ for any $0<p<2/3$. Note that on the other hand we do have a uniform positive lower bound for the measure of the set $\{\vp\leq\inf_M\vp+1\}$, see \eqref{measure}.

The bound \eqref{bd0} was established in \cite{TWY}  for the function $\mathcal{F}(x)=e^{\alpha x}$ using an estimate of the trace of $\tilde{\omega}$  \cite{W, TWY}. In fact it is possible to prove (\ref{bd0}) in general using the arguments of \cite{TWY, TW2}  (see in particular Remark 3.1 of \cite{TW2}). However, the proof we give here using the ABP estimate, which was pointed out to us by Sz\'ekelyhidi \cite{Sz2}, 
  is much simpler and we believe it is more natural. 
  
 It is perhaps surprising that the ABP method works here for $\varphi$, even though the equation (\ref{CYeqn}) is not a  local scalar PDE involving $\varphi$. We remark that the proof of Theorem \ref{theorem2} does not actually require the assumption that $[\ti{\omega}]=[\omega]$, and so this condition (or something similar, see \cite[Question 2.1]{TW4}) must be used in any proof of the missing integral bound of $\varphi$.

Finally we remark that Theorem \ref{theorem2} holds in higher dimensions with  almost exactly the same proof, after making a suitable change to the inequality $q>3/2$.

The outline of the paper is as follows.  In Section \ref{sectionABP} we outline the ABP method of Sz\'ekelyhidi in a rather general setting.  We then use it to prove Theorem \ref{theorem1} in Section \ref{sectionKT} and Theorem \ref{theorem2} in Section \ref{sectionCYs}.

\bigskip
\noindent
{\bf Acknowledgements.} \ The authors thank G. Sz\'ekelyhidi for many useful conversations.

\section{The ABP method of Sz\'ekelyhidi} \label{sectionABP}

Fix a ball $B=B_r(0) \subset \mathbb{R}^n$ centered at the origin of radius $r$ with $0<r\le 1$.
In \cite{Sz}, the following ABP estimate is proved for a smooth function $v: \ov{B} \rightarrow \mathbb{R}$.  (It is stated there for $r=1$ but this can be trivially extended to $0<r\le 1$.)

\begin{proposition} \label{gabor}
Assume that for $\ve>0$, the function $v: \ov{B}\rightarrow \mathbb{R}$ satisfies $v(0) + \ve \le \inf_{\partial B} v$.  If we define
$$P = \{ x \in B \ | \ |Dv(x)| < \ve/2,  \ \textrm{and } v(y) \ge v(x) + Dv(x) \cdot (y-x) \ \forall y \in B \},$$
then
$$ \ve^n \le  C_0 \int_P \det (D^2 v),$$
for a constant $C_0$ depending only on the dimension $n$.
\end{proposition}

The difference with the classical ABP estimate \cite[Lemma 9.2]{GT} is that $P$ is a subset of those points where $v$ has small derivative, and this is crucial for our application.
We now describe Sz\'ekelyhidi's method \cite{Sz} for applying this estimate to prove $L^{\infty}$ bounds (cf. \cite{B}).  Let $M$ be a compact manifold of real dimension $n$ with a fixed volume form $d\mu$.
Let $u: M \rightarrow \mathbb{R}$ be a smooth function with $\sup_M u=0$.    The function $u$ will satisfy a PDE but to keep the discussion general we will not specify the equation.
 We wish to obtain an upper bound for $\| u \|_{L^{\infty}(M)}$ in terms of an integral bound for $u$.  We will reduce this to a key pointwise inequality on $P$, for $B$ and $v$ which we will now specify.

  Suppose $u$ achieves its infimum at $x_0 \in M$ (assume without loss of generality that $\inf_M u < -1$) and take a coordinate chart centered at $x_0$ which we identify with a ball $B=B_r(0)$ of a fixed radius $0<r\le 1$.  Consider the function $v$ on $B$ defined by
$$v = u + \frac{\ve}{r^2} \sum_{i=1}^n x_i^2,$$
for a small uniform fixed $\ve>0$.  Applying Proposition \ref{gabor} we obtain
$$\ve^n \le C_0 \int_P \det (D^2v).$$
The {\bf key estimate} we need to prove, which will use the PDE satisfied by $u$ and the definition of $v$ and $P$, is the following:
\begin{equation} \label{keyestimate}
\textrm{at every $x\in P$ we have $\det (D^2v(x)) \le C$, for uniform $C$.}
 \end{equation}
Assume now that (\ref{keyestimate}) holds.  Then it follows that
\begin{equation}\label{Plb}
\ve^n \le C |P|,
\end{equation}
up to increasing the uniform constant $C$, where $| \cdot |$ denotes the measure with respect to the volume form $d\mu$.

Now the integral bound for $u$ comes in.  Let $\mathcal{F}:\mathbb{R} \to \mathbb{R}_{\geq 0}$ be an increasing function with $\lim_{x\to+\infty}\mathcal{F}(x)=+\infty$.  By definition, on the set $P$ we have $v \le v(0) + \ve/2$ and hence $u \le \inf_M u + \ve/2$.   For later use, we note that this with (\ref{Plb}) implies \begin{equation} \label{measurelowerbound}
| \{ u \le \inf_M u +1 \} | \ge \frac{\ve^n}{C}.
\end{equation}
Since $\mathcal{F}$ is increasing, on $P$ we have $\mathcal{F}(-u) \ge \mathcal{F}(-\inf_M u - \ve/2)$ and so
$$\ve^n \le C |P| \le C \frac{\int_M \mathcal{F}(-u) d\mu}{\mathcal{F}( - \inf_M u - \ve/2)}.$$
  Hence we obtain
$$\mathcal{F}( \| u \|_{L^{\infty}} - 1) \le \frac{C}{\ve^n} \int_M \mathcal{F}(-u)d\mu.$$
Since $\ve>0$ is uniform, it follows that a bound on $\int_M \mathcal{F}(-u)d\mu$ gives a bound on $\| u\|_{L^{\infty}(M)}$.  In the special case $\mathcal{F}(t)=t^p$ for $p>0$,  an $L^p$ bound of $u$ gives an $L^{\infty}$ bound of $u$.

In each of the two applications below, we will show  that the key estimate (\ref{keyestimate}) holds.

\section{The Kodaira-Thurston manifold with $S^1$ symmetry} \label{sectionKT}

We give the proof of Theorem \ref{theorem1}.  Let $M$ be the Kodaira-Thurston manifold, as described in the introduction.
We follow the notation in Buzano-Fino-Vezzoni \cite{BFV}, defining
$$e^1=dy, \quad e^2=dx, \quad e^3=dt, \quad e^4=dz-xdy,$$
where $t \in [0,2\pi)$ is the variable in the $S^1$ factor.  Consider the almost complex structures $J_{(1)}$ and $J_{(2)}$, defined by
$$J_{(1)}e^1=e^3,\quad J_{(1)}e^4=e^2,$$
$$J_{(2)}e^1=e^4,\quad J_{(2)}e^2=e^3.$$
These give rise to an $S^1$ family $\{ J_{\theta}\}_{\theta \in [0,2\pi)}$ of almost complex structures given by
$$J_{\theta} = \cos \theta \, J_{(1)} + \sin \theta \, J_{(2)}.$$
Each $J_{\theta}$ is compatible with the symplectic form
$$\omega_\theta=(\cos\theta\,e^1+\sin\theta\,e^2)\wedge e^3+e^4\wedge(-\sin\theta\,e^1+\cos\theta\,e^2),$$
and the resulting almost-K\"ahler metric is independent of $\theta$ and equal to
$$g=(e^1)^2+(e^2)^2+(e^3)^2+(e^4)^2.$$
$M$ does not admit a K\"ahler structure, and so none of the $J_{\theta}$ are integrable.  Note that for all $\theta$, the data $(M, \omega_\theta, J_\theta)$  are invariant under the $S^1$ action on $M$ given by translation in the $z$ coordinate.

\begin{remark}
From our earlier work \cite{TW}  one can solve the Calabi-Yau equation (\ref{CYeqn}) on $(M, \omega_{\theta}, J_{\theta})$ for any function $F$ which is invariant under the $T^2$ action which translates $z$ and $t$ (in \cite{TW} we considered the case $\theta=\frac{\pi}{2}$, but the same argument applies to all values of $\theta$).
\end{remark}

\begin{remark}  The manifold $M$ also admits the integrable complex structure $J_{(3)}=J_{(1)}J_{(2)}$ given by
$$J_{(3)}e^1=e^2,\quad J_{(3)}e^3=e^4,$$
which makes $M$ into a primary Kodaira surface. These three almost complex structure satisfy the quaternionic relations, and all invariant almost complex structures on $M$ which are compatible with $g$ and its orientation are of the form
$aJ_{(1)}+bJ_{(2)}+cJ_{(3)}$, with $a,b,c\in\mathbb{R}, a^2+b^2+c^2=1$; in this paper we are choosing $c=0$.
\end{remark}

The vector fields
$$X=\cos\theta\,\de_x-\sin\theta(\de_y+x\de_z), \ Y=\sin\theta\,\de_x+\cos\theta(\de_y+x\de_z),\ \de_t, \ \de_z,$$
are the dual frame of the coframe
$$ -\sin\theta\,e^1+\cos\theta\, e^2, \  \cos\theta\,e^1+\sin\theta\,e^2,\  e^3, \ e^4.$$
When $X$ and $Y$ act on $S^1$-invariant functions they are simply equal to
$$X=\cos\theta\,\de_x-\sin\theta\,\de_y, \quad Y=\sin\theta\,\de_x+\cos\theta\,\de_y,$$
and now these commute with each other, and with $\de_t$.
We will use $\partial_X$ and $\partial_Y$ to denote derivatives with respect to these vector fields.

We recall from Buzano-Fino-Vezzoni \cite[(61)]{BFV} that the Calabi-Yau equation on the Kodaira-Thurston manifold with $S^1$ symmetry can be reduced to the following elliptic PDE for a function $u$ on the three-torus $T^3$ (viewed as a quotient $\mathbb{R}^3/\mathbb{Z}^3$ with coordinates $(x,y,t)$):
\begin{equation} \label{PDE}
(u_{XX}+1)(u_{YY}+u_{tt}+u_t+1) - u_{XY}^2 - u_{Xt}^2 = e^F,
\end{equation}
with \cite[Proposition 1]{BFV}
\begin{equation} \label{elliptic}
u_{XX}+1>0, \quad u_{YY}+u_{tt}+u_t+1>0.
\end{equation}
We make the normalization $\sup_M u=0$.
Following  \cite{BFV}, it is enough to prove a uniform $L^\infty$ estimate for $u$.  Indeed, the key second order estimate $\sup_M |\Delta u| \le C(1+ \sup_M |\nabla u|)$ \cite[Theorem 6]{BFV} requires only the $L^{\infty}$ bound for $u$, and the remaining $C^{\infty}$ estimates are a consequence of this.  Note that although this second order estimate is proved by Buzano-Fino-Vezzoni under the assumption $\theta=0$, the proof for general $\theta \in [0,2\pi)$ is formally the same, as pointed out in Section 5 of  \cite{BFV}.

First note that adding the inequalities (\ref{elliptic}) gives that $\Delta u + u_t > -2$, for $\Delta$ the Laplace operator of the flat metric on the torus.  Then a Green's function argument as in \cite[Proposition 2.3]{CTW} gives a uniform $L^1$ bound on $u$ (or we can also use the weak Harnack inequality as in \cite[(43)]{Sz} to obtain an $L^p$ bound for some $p>0$).

We now complete the argument for the $L^{\infty}$ bound of $u$ following the ABP method, as outlined in Section \ref{sectionABP}.
We translate the coordinates so that $\inf_M u$ is achieved at the origin and we work in a ball $B=B_r$ of a fixed radius $r=1/4$ say, centered at $0$.
Define
$$v = u + \frac{\ve}{r^2} (x^2 + y^2 + t^2)$$
for a small $\ve>0$.
We use the terminology of Section \ref{sectionABP}.  It is sufficient to establish (\ref{keyestimate}), namely that at every point $x \in P$ we have
$$\det D^2 v(x) \le C.$$
By definition of $P$ we may assume $D^2 v \ge 0$ and $|Dv|\le \ve/2$ at $x$.
We can rewrite the PDE (\ref{PDE}) as
\begin{equation} \label{PDEv}
(v_{XX} + 1- O(\ve) ) (v_{YY} + v_{tt} + 1 - O(\ve)) - v_{XY}^2 - v_{Xt}^2 = e^F,
\end{equation}
noting that from $D^2v\ge 0$ we have $v_{XX} , v_{YY}, v_{tt} \ge 0$.
Hence, as long as $\ve$ is sufficiently small,
$$\frac{1}{2} \left(v_{XX} + v_{YY} + v_{tt} + \frac{1}{2}\right) + v_{XX} v_{YY} + v_{XX} v_{tt} \le e^F + v_{XY}^2 + v_{Xt}^2.$$
But since $D^2v\ge 0$ we have $v_{XY}^2 \le v_{XX} v_{YY}$ and $v_{Xt}^2 \le v_{XX} v_{tt}$ and it follows that
$$v_{YY} + v_{tt} +  v_{XX} +\frac{1}{2} \le 2e^F.$$
Hence $\textrm{tr}(D^2v) \le C$ and by the arithmetic-geometric mean inequality we obtain the required bound $\det D^2 v\le C$.  Since we have already established the $L^1$ bound of $u$, this completes the proof of the $L^{\infty}$ estimate and Theorem \ref{theorem1}.

\section{The Calabi-Yau equation on symplectic 4-manifolds} \label{sectionCYs}

We now give the proof of Theorem \ref{theorem2}, so we assume we are in that setting with $\tilde{\omega}$ solving the Calabi-Yau equation (\ref{CYeqn}) and $\vp$ the almost-K\"ahler potential given by (\ref{akp}).  Recall from \cite{W, TWY} that to find $\varphi$ given $\omega, \tilde{\omega}$,  solve the Poisson equation $\tilde{\Delta} \varphi : = \tr{\tilde{\omega}}(\frac{1}{2} dJd\varphi) = 2 - \tr{\tilde{\omega}}{\omega}$ subject to $\sup_M \varphi=0$, and then define the $1$ form $a$ by $da = \tilde{\omega} - \omega - \frac{1}{2} dJd\varphi$.  The function $\varphi$ is unique, while $a$ is unique up to ``gauge''.  Here $\tr{\tilde{\omega}}{\omega}$ is defined to be  $2(\tilde{\omega} \wedge \omega )/ \tilde{\omega}^2$.

We begin with the proof of (\ref{bd0}).  As noted in the introduction, this argument is due to Sz\'ekelyhidi. Following again the method of Section \ref{sectionABP}, we work in a coordinate chart, identified with the unit ball $B$,
in  which $\varphi$ achieves its $\inf_M\varphi$ at the origin.  Define $v= \varphi + \ve |x|^2$.  We need to prove (\ref{keyestimate}).  Let $x$ be a point
 where $D^2 v \ge 0$ and $|Dv|\le \ve/2$.  It suffices to show that $\det D^2 v$ is bounded from above at this point.

 Note that, writing $(D^2v)^J$ for the $J$-invariant part of $D^2v$, given by
\begin{equation} \label{D2vJ}
(D^2v)^J = \frac{1}{2} (D^2v + J^T (D^2v) J),
\end{equation}
we have that, at $x$,
\begin{equation} \label{dJdv11}
(dJd\varphi)^{(1,1)} = (D^2v)^J + O(\ve) \ge O(\ve),
\end{equation}
e.g. thanks to \cite[p.443]{TWWY}.  Note that here we are using the condition $|Dv| \le \ve/2$. It follows that
$$\omega^{(1,1)} + \frac{1}{2}(dJd\varphi)^{(1,1)} \ge \frac{ \omega^{(1,1)}}{2},$$
provided we choose $\ve$ sufficiently small (but uniform).
If we wedge this with $\tilde{\omega}$ we get
\[\begin{split}
\ti{\omega}^2&=\ti{\omega}\wedge\left(\omega^{(1,1)}+\frac{1}{2}(dJd \varphi)^{(1,1)}\right)\geq\frac{ \ti{\omega}\wedge\omega}{2},
\end{split}\]
since $\ti{\omega}\wedge da=0$ and $\ti{\omega}$ is of type $(1,1)$, and so
$$\tr{\ti{\omega}}{\omega}\leq 4.$$
By the Calabi-Yau equation (\ref{CYeqn}), this implies that $\omega^{(1,1)}$ and $\tilde{\omega}$ are uniformly equivalent.
But from (\ref{dJdv11}) and
\begin{equation} \label{tiDvp}
\ti{\Delta}\vp= \frac{1}{2}\tr{\tilde{\omega}}{ (dJd\varphi)^{(1,1)}}=  2-\tr{\ti{\omega}}{\omega}\leq 2,
\end{equation}
 we see  that $(dJd\varphi)^{(1,1)}$ is uniformly bounded.  We can then argue exactly  as in \cite{CTW} to obtain an upper bound on $\det D^2v$ (see also \cite{B, Sz}).  Indeed, using (\ref{D2vJ}), (\ref{dJdv11}) and  the inequality $\det (A+B) \ge \det A + \det B$ for nonnegative matrices $A,B$, we have
$$\det D^2v \le 8 \det ((D^2v)^J) \le C,$$
as required.
 This completes the proof of (\ref{bd0}).

Furthermore, recalling the estimate (\ref{measurelowerbound}) of
Section \ref{sectionABP}, we have
\begin{equation} \label{measure}
|\{ \varphi \le \inf_M \varphi +1 \}| \ge \delta,
\end{equation}
for a uniform $\delta>0$, where $| \cdot |$ is the measure of the set with respect to the volume form $\omega^2$.

\begin{remark} In fact, a simple modification of these arguments shows that (\ref{bd0}) holds for the ``almost-K\"ahler potentials'' $\vp_t$ for $\frac{1}{2}<t\leq 1$ (as defined in \cite{W}), when $\omega$ is also assumed to be compatible with $J$. On the other hand $\vp_{\frac{1}{2}}$ is uniformly bounded, as observed by Donaldson using Moser iteration (see \cite[Remark 6.1]{W}).
\end{remark}

To finish the proof of Theorem \ref{theorem2}, it suffices to prove the following:
\begin{proposition}\label{main3}
Given any $0<p<2/3$, the function $\vp$ satisfies the estimate
\begin{equation}\label{bd3}
\|\vp-\inf_M\vp\|_{L^p(M)}\leq C_p,
\end{equation}
for a uniform constant $C_p$ which depends only on the background data and on $p$.
\end{proposition}

Indeed, suppose that Proposition \ref{main3} holds, and let $\lambda>0$. We may assume without loss of generality that $\lambda <-\inf_M \varphi$. Then  for any $0<p<2/3$,
$$(- \lambda - \inf_M \varphi)^p | \{ \varphi \ge - \lambda \} | \le \int_{ \{ \varphi \ge -\lambda\} } (\varphi - \inf_M \varphi)^p \omega^2 \le C_p,$$
and then (\ref{mi}) holds with $q=1/p>3/2$ and $C_q = C_p^{1/p}$.

\begin{proof}[Proof of Proposition \ref{main3}]
Using the elementary formula
$$\int_M u^p \omega^2 =p\int_0^\infty |\{u\geq s\}|s^{p-1}ds,$$
applied to $u=\vp-\inf_M\vp$,
it is enough to show that
\begin{equation}\label{good}
|\{\varphi\geq \inf_M\varphi +s\}|\leq C s^{-\frac{2}{3}}.
\end{equation}
for a uniform constant $C>0$ and for all $s\geq 2.$

Let $\psi_s=\max(-\varphi+\inf_M\varphi +s,0)$. Then $\psi_s$ is a Lipschitz function that satisfies $0\leq \psi_s\leq s$. It is a well-known fact that on the set $\{\psi_s=0\}$ the gradient of $\psi_s$ is zero
a.e. (see e.g. \cite[Theorem 3.2.6]{AT}). We have an obvious pointwise bound
$$|\de \psi_s |^2_{\omega}\leq \tr{\omega}{\ti{\omega}}|\de\psi_s|^2_{\ti{\omega}},$$
where we define
$$|\de \psi_s |^2_{\omega} = \frac{\omega \wedge d\psi_s \wedge Jd \psi_s}{\omega^2}.$$
We also have the  $L^1$ bound $\int_M (\tr{\omega}{\ti{\omega}})  \omega^2\leq C$.  Then, using in addition the Calabi-Yau equation (\ref{CYeqn}),
\begin{equation}\label{a}
\begin{split}
\int_M|\de\psi_s|_{\omega}\omega^2 \le {} & \left(\int_M (\tr{\omega}{\ti{\omega}})\omega^2\right)^{\frac{1}{2}}\left(\int_M |\de\psi_s|^2_{\ti{\omega}} \omega^2\right)^{\frac{1}{2}} \\
\leq {} & C\left(\int_M|\de\psi_s|^2_{\ti{\omega}} \ti{\omega}^2\right)^{\frac{1}{2}} \\
= {} & C\left(\int_{\{\psi_s>0\}}|\de\psi_s|^2_{\ti{\omega}} \ti{\omega}^2\right)^{\frac{1}{2}}.
\end{split}
\end{equation}
Since $\psi_s$ is smooth on the open set $\{\psi_s>0\}$ and vanishes on its boundary, we can integrate by parts (this is justified by exhausting this set by smooth domains) and get, using (\ref{tiDvp}) and again \eqref{CYeqn},
$$\int_M|\de\psi_s|_{\omega}\omega^2\leq C\left(\int_{\{\psi_s>0\}}\psi_s\ti{\Delta}(-\psi_s)\ti{\omega}^2\right)^{1/2}\leq C
\left(\int_{\{\psi_s>0\}}\psi_s \omega^2\right)^{1/2}\leq Cs^{\frac{1}{2}}. $$
Now the Sobolev-Poincar\'e inequality gives
\begin{equation} \label{SP}
\|\psi_s-\underline{\psi_s}\|_{L^{\frac{4}{3}}(M)}\leq C\|\de\psi_s\|_{L^1(M)}\leq Cs^{\frac{1}{2}},
\end{equation}
where we define $\underline{\psi_s}=\frac{\int_M\psi_s\omega^2}{\int_M\omega^2}$.
Let $$\gamma(s)=|\{\psi_s=0\}|=|\{\varphi\geq\inf_M\varphi+s\}|,$$
so that
\begin{equation}\label{b}
\|\psi_s-\underline{\psi_s}\|_{L^{\frac{4}{3}}(M)}\geq \left(\int_{\{\psi_s=0\}}
|\psi_s-\underline{\psi_s}|^\frac{4}{3}\omega^2\right)^{\frac{3}{4}}=\gamma(s)^{\frac{3}{4}}
\underline{\psi_s}.
\end{equation}
Then from \eqref{SP} and \eqref{b} we get
\begin{equation}\label{stima}
\gamma(s)^{\frac{3}{4}} \underline{\psi_s}\leq Cs^{\frac{1}{2}}.
\end{equation}
Now we look at the set $\{\varphi\leq\inf_M\varphi+1\}$
which by \eqref{measure} has volume at least $\delta$. On this set we have
that $\psi_s\geq s-1 \geq s/2,$ as long as $s\geq 2$. Then we have
$$\underline{\psi_s} \int_M \omega^2= \int_M \psi_s\omega^2\geq\int_{\{\varphi\leq\inf_M\varphi+1\}}\psi_s\omega^2\geq \frac{\delta s}{2}.$$
Combining this with \eqref{stima} we finally get
$$\gamma(s)\leq C s^{-\frac{2}{3}},$$
completing the proof of the proposition.
\end{proof}

\end{document}